\documentclass[12pt,a4paper,reqno]{amsart}
\usepackage[colorlinks,pdfstartview=FitH,citecolor=blue,linkcolor=red,urlcolor=blue]{hyperref}

\usepackage[noadjust]{cite}
\usepackage{tikz}\usetikzlibrary{matrix,decorations.pathreplacing,positioning}
\usepackage{xcolor}
\usepackage{footnote}
\usepackage{hyperref}
\usepackage{amsmath,amsthm,amssymb,bm}
\usepackage{calrsfs}
\usepackage{rotating}

\usepackage{multirow}
\usepackage{array}
\newcolumntype{x}[1]{>{\centering\arraybackslash\hspace{0pt}}p{#1}}
\usepackage{empheq}
\usepackage{wasysym}
\usepackage{caption}
\usepackage{subcaption}

\usepackage{enumitem}
\setitemize{itemsep=-1pt}
\setenumerate{itemsep=-1pt}

\usepackage[margin=2.6cm]{geometry}
\usepackage{pgfplots}
\pgfplotsset{width=10cm,compat=1.9}

\definecolor{myg}{RGB}{220,220,220}
\usepackage{etoolbox}
\AtEndEnvironment{example}{\null\hfill$\blacktriangle$}%
\usepackage{caption}

\newtheorem{theorem}{Theorem}[section]
\newtheorem{corollary}[theorem]{Corollary}
\newtheorem{proposition}[theorem]{Proposition}
\newtheorem{lemma}[theorem]{Lemma}

\newtheorem{definition}[theorem]{Definition}

\newtheorem{remark}[theorem]{Remark}

\newcommand*{\myproofname}{Proof of the claim}

\newcommand{\numberset}{\mathbb}

\newcommand{\R}{\numberset{R}}

\newcommand{\F}{\numberset{F}}

\newcommand{\mC}{\mathcal{C}}

\newcommand{\mA}{\mathcal{A}}

\newcommand{\mB}{\mathcal{B}}

\newcommand{\Fq}{\F_q}

\newcommand{\colsp}{\textnormal{colsp}}
\newcommand{\rowsp}{\textnormal{rowsp}}

\newcommand\qqbin[2]{\left[\begin{matrix} #1 \\ #2 \end{matrix} \right]}

\DeclareMathOperator{\dd}{d}

\newcommand{\rk}{\textnormal{rk}}

\DeclareMathOperator{\GL}{GL}

\DeclareMathOperator{\Alt}{Alt}

\DeclareMathOperator{\Mat}{Mat}

\newlength{\mynodespace}
\allowdisplaybreaks

\title{Eigenvalue bounds and alternating rank-metric codes}

\usepackage[foot]{amsaddr}
\author{Aida Abiad$^{1}$}
\author{Gianira N. Alfarano$^{1,2}$}
\author{Alberto Ravagnani$^1$}
\address{$^1$Department of Mathematics and Computer Science, Eindhoven University of Technology, The Netherlands.} 
\address{$^2$School of Mathematics and Statistics, University College Dublin, Ireland.}
\email{a.abiad.monge@tue.nl, gianira.alfarano@gmail.com, a.ravagnani@tue.nl}

\date{}

\begin{document}

\thispagestyle{empty}

\begin{abstract}
   In this note we apply a spectral method to the graph of alternating bilinear forms. In this way, we obtain upper  bounds on the size of an alternating rank-metric code for given values of the minimum rank distance. We computationally compare our results with Delsarte's linear programming bound, observing that they give the same value. For small values of the minimum rank distance, we are able to establish the equivalence of the two methods. The problem remains open for larger values.
\end{abstract}
\maketitle

\section{Introduction}

A fundamental problem in coding theory is to determine the largest possible size of an error-correcting code with given minimum distance. This problem has been investigated for several distance functions, most notably for the Hamming distance.

In his PhD dissertation, Delsarte proposed a linear programming (LP) method to derive bounds for the cardinality of an error-correcting code based on association schemes. The LP method applies to several distance functions and often performs very well. 
One potential drawback of Delsarte's approach is the necessity to explicitly compute the parameters of the underlying association scheme, a task that is not always easy. The parameters have been successfully computed for the graphs associated with various metrics, including the Hamming metric~\cite{delsarte1973algebraic}, the rank metric~\cite{D1978,delsarte1975alternating}, 
permutation codes \cite{DIL2020}, and the Lee metric \cite{A1982,S1986}, 
but also for newer metrics such as the subspace distance \cite{R2010}.

In recent years, it was observed that spectral graph theory can also be applied to derive bounds on the cardinality of error-correcting codes, sometimes outperforming other approaches, see \cite{AKR2024,ANR2024}. Since methods based on spectral graph theory are often easier to apply than Delsarte's LP method,
it is natural to ask how they compare to each other. This is precisely the question we address in this note, for the case of alternating rank-metric codes.

This note focuses on the case of codes made of alternating matrices
endowed with the rank distance, from the algebraic perspective. For this metric, we provide evidence that the recent spectral graph theory bounds from \cite{ACF2019} (the so-called Ratio-type bounds) and their LP implementation \cite{F2020} are equally good as Delsarte's LP approach \cite{delsarte1975alternating}, although much easier to apply. We also show that the classical coding approaches for bounding the size of  alternating rank-metric codes do not provide any improvement compared with the algebraic approaches.

The rest of the paper is organized as follows. Section \ref{sec:preliminaries} presents the needed preliminaries. Section \ref{sec:neweigenvaluebounds} contains the core of the paper: we study the graph of alternating bilinear forms in detail, we apply the Ratio-type eigenvalue bounds for the $k$-independence number of this graph and we compare the output with Delsarte's LP output. In Section \ref{sec:newcodingbounds} we use a number of methods from coding theory in order to derive bounds on the size of alternating rank-metric codes. We conclude with some remarks in Section \ref{sec:conclusion}.


\section{Preliminaries}\label{sec:preliminaries}
 
Throughout the paper, $q$ denotes a prime power and $\Fq$ the finite field with~$q$ elements. For a positive integer $n$,
we denote by $\Mat_{n}(\Fq)$ the space of $n\times n$ matrices with entries in $\Fq$. For a given matrix $M\in\Mat_{n}(\Fq)$, we denote by $\colsp(M)$ the subspace of $\F_q^n$ spanned by the columns of $M$ and by $\rowsp(M)$ the subspace of $\F_q^n$ spanned by the rows of $M$. We denote by $O$ the $n\times n$ zero-matrix.  The standard basis vector of $\Fq^n$ is denoted by $\{e_1,\ldots,e_n\}$.

\subsection{Alternating Rank-Metric Codes}
In this subsection we  give the needed definitions on alternating rank-metric codes.

\begin{definition}
A matrix $A\in\Mat_{n}(\F_q)$ is said to be \textbf{alternating} or \textbf{skew-symmetric} if $A=-A^\top$ and $A_{i,i}=0$ for all $i\in\{1,\ldots,n\}$. We denote by $\Alt_{n}(\Fq)$ the space of $n\times n$ alternating matrices with entries in $\Fq$.
\end{definition}

 We endow the space $\Alt_n(\F_q)$ with the \textbf{rank-metric}, which is defined as follows
 $$\rk(A,B) = \mathrm{rk}\,(A-B), \qquad \mbox{ for } A,B\in\Alt_n(\F_q).$$
 The \textbf{minimum rank distance} of $\mC$ is defined as
$$d := \dd_\rk(\mC) =\min\{ \rk(A,B) \colon A,B \in \mC,\,\, A\neq B \}.$$
We say that a subset of $\Alt_{n}(\F_q)$ endowed with the rank distance is an \textbf{alternating (rank-metric) code} or a  \textbf{skew-symmetric (rank-metric) code}.
If $\mC$ is a $k$-dimensional $\F_q$-linear subspace  of $\Alt_n(\F_q)$, we say that $\mC$ is a \textbf{linear} alternating code. 

Notice that alternating matrices have always even rank, hence we have that the minimum distance of an alternating code is always even and we usually assume that it is $2d$, for some positive integer $d$. 
We use $A_q(n, 2d)$ to denote the maximal cardinality of an alternating code in $\Alt_n(q)$ with the property that any two distinct codewords are at distance at least $2d$ from each other.

In \cite{delsarte1975alternating}, Delsarte and Goethals considered precisely the problem of finding
an upper bound on $A_q(n,2d)$ using Delsarte's linear programming (LP) method; see Section \ref{sec:Delsarte} for a brief explanation and \cite{delsarte1973algebraic,delsarte1998association} for a detailed treatment. This resulted in the following bound relating the parameters of an alternating code. Due to the similarity with the well-known Singleton bound for rank metric-codes (\cite{D1978}), we refer to it as Singleton-like bound for alternating codes.

\begin{theorem}[\cite{delsarte1975alternating}]\label{thm:del-goeth}
Let $\mC\subseteq  \Alt_{n}(\F_q)$ be an alternating code. Then 
\begin{equation}\label{eq:SingletonBound}
    A_q(n,2d)\leq q^{\frac{n(n-1)}{2\left\lfloor \frac{n}{2}\right\rfloor}\left( \left\lfloor \frac{n}{2}\right\rfloor-d+1 \right)}.
\end{equation}
\end{theorem}
We call the alternating codes attaining the Singleton-like bound from Eq.~\eqref{eq:SingletonBound} \textbf{alternating MRD codes}. 

Note that $2 \left\lfloor \frac{n}{2}\right\rfloor$ is the maximum rank that a matrix $M\in\Alt_{n}(\F_q)$ can have. Moreover the quantity $\frac{n(n-1)}{2\left\lfloor \frac{n}{2}\right\rfloor}$ is $n$ or $n-1$ according to the parity of $n$.

\begin{remark}
Delsarte and Goethals \cite{delsarte1975alternating} also showed that alternating MRD codes exist for $n$ odd and any $q$, or for $n$ even and $q$ even. Showing the existence of alternating MRD codes for $n$ even and $q$ odd is still an open question.
\end{remark}

\subsection{Algebraic Graph Theory Tools} In this subsection we provide some graph theory definitions which will be used in Section \ref{sec:neweigenvaluebounds}.

Let $\Gamma = (V,E)$ be a graph with adjacency matrix $A$, having distinct eigenvalues $\theta_0 > \cdots > \theta_r$. 
Let $u,v$ be two vertices in $V$. We define the \textbf{geodesic distance} between $u$ and $v$ as the minimum length of a path joining them and we denote it by  $\dd_\mathrm{g}(u,v)$.
\begin{definition}
    A connected graph $\Gamma$ is called  \textbf{distance-regular} if it is regular of degree $\delta$ and if for every two vertices $u,v$ such that $\dd_\mathrm{g}(u,v)=i$ there are exactly $c_i$ neighbours $z$ of $v$ such that $\dd_\mathrm{g}(z,u)=i-1$ and there are exactly $b_i$ neighbours $z$ of $v$ such that $\dd_\mathrm{g}(z,u)=i+1$.
\end{definition}

The sequence 
$$\iota(\Gamma):=\{b_0,\ldots,b_{D(\Gamma)}; c_1,\ldots, c_{D(\Gamma)}\},$$
where $D(\Gamma)$ is the diameter of $\Gamma$, is called the \textbf{the intersection array} of $\Gamma$. 
Let
$$a_i = \delta - b_i-c_i, \quad \textnormal{ for } i=0,\ldots, D(\Gamma).$$
The numbers $a_i, b_i, c_i$ are called the \textbf{intersection numbers} of $\Gamma$. Note that $a_i$ represents the number of neighbours $z$ of $v$ such that $\dd_\mathrm{g}(z,u)=i$, for every $u,v\in V$ such that  $\dd_\mathrm{g}(u,v)=i$.

\begin{definition}
    A graph $\Gamma$ is \textbf{strongly regular} if it is regular and for some given integers $\lambda,\mu\geq 0$, every two adjacent vertices of $\Gamma$ have $\lambda$ common neighbours, and every two non-adjacent vertices of $\Gamma$ have $\mu$ common neighbours.
\end{definition}

\begin{definition}
A graph is \textbf{$l$-partially walk-regular} if for any vertex~$v$ and any positive integer $i\leq l$ the number of closed walks of length~$i$ that start and end in~$v$ does not depend on the choice of~$v$. A graph is \textbf{walk-regular} if it is \mbox{$l$-partially} walk-regular for any positive integer~$l$.
\end{definition}

\begin{definition}
For $k\ge 1$, the $k$-\textbf{independence number}  of a graph $\Gamma$, denoted $\alpha_k(\Gamma)$, is the maximum number of vertices that are mutually at (geodesic) distance $\dd_\mathrm{g}$ greater than~$k$. Note that for $k=1$ we obtain the classical independence number of a graph,~$\alpha(\Gamma)$. 
\end{definition}

\subsection{Delsarte's Linear Programming}\label{sec:Delsarte}
In this subsection we give a short overview of Delsarte's linear program. For a detailed treatment, we refer the interested reader to~\cite{delsarte1973algebraic, delsarte1998association}.

Let $\Gamma$ be a regular graph with diameter $D$ and let $\mC\subseteq V(\Gamma)$. The \textbf{distance distribution} of $\mC$ is the vector $(\beta_0,\ldots,\beta_D)$ where $\beta_i$ is the mean number of vertices in $\mC$ at distance $i$ in $\Gamma$ from a given vertex $u\in\mC$, i.e.
$$\beta_i=\frac{1}{|\mC|}\sum_{u\in\mC}|\Gamma_{i}(u)\cap \mC|,$$
where $\Gamma_i(u)=\{v\in\mC| \dd_\mathrm{g}(u,v)=i\}$.
Consider $\mC$ to be a $(d-1)$-independent set in $\Gamma$, then, clearly, $\beta_0=1$ and $\beta_i=0$ for $1\leq i\leq d-1$. Delsarte proved that the distance distribution of any code must satisfy certain inequalities expressed in terms of the $Q$-numbers of the Bose-Mesner algebra of $\Gamma$. 

Note that $\sum_{i=0}^D\beta_i=|\mC|$. Therefore, to upper bound the number of codewords, the objective
of the linear program is to maximize $\sum_0^D\beta_i$. Hence, the linear program results as in Eq.~\eqref{DelsarteLP}.
\begin{equation}\label{DelsarteLP}
\boxed{
\begin{array}{ll@{}ll}
\text{maximize}  &\sum_{i=0}^D \beta_i &\\
\text{subject to} &\sum_{i=0}^D\beta_iQ_j(i)\geq 0, & \quad\textnormal{ for all } 1\leq j\leq k\\
&\beta_0=1 & \\
&\beta_i=0, & i=1,\ldots,d-1\\
&\beta_i\geq0, &\quad i=d,\dots,D\\
\end{array}
}
\end{equation}

\section{The graph $\Gamma(\Alt_n(\F_q))$ and  eigenvalue bounds}\label{sec:neweigenvaluebounds}
In this section we introduce and study spectral properties of the alternating bilinear forms graph, denoted $\Gamma(\Alt_n(\F_q))$. Later on we will use these properties to derive the eigenvalue upper bounds on the size of alternating rank-metric codes.

\begin{definition}
Let $\Gamma(\Alt_n(\F_q))$ be the \textbf{alternating bilinear forms graph}, i.e. the graph whose vertices are the distinct matrices of $\Alt_n(\F_q)$ and whose edges are the pairs $(A,B)\in (\Alt_n(\F_q))^2$ with the property that $\rk(A-B)=2$. 
\end{definition}

Before we can derive the main results of this section, we need some preliminary results.
The following is well-known. 

\begin{lemma}\label{lem:sumrank2}
    Let $A\in\Alt_n(\F_q)$ be an alternating matrix with rank $2r$ for some integer~$r$. Then $A$ can be written as sum of $r$ many rank $2$ alternating matrices.
\end{lemma}
\begin{lemma}\label{lem:geo_distance}
    The geodesic distance between two vertices $X,Y$ in $\Gamma(\Alt_n(\F_q))$ coincides with the rank distance between $X$ and $Y$ divided by $2$.
\end{lemma}
\begin{proof}
    Let $X,Y\in\Alt_n(\F_q)$ be such that $\rk(X-Y)=2r$ for some $0\leq r\leq \lceil\frac{n}{2}\rceil$. Then, by Lemma \ref{lem:sumrank2} there exist $r$ matrices $A_i\in\Alt_n(\F_q)$ of rank $2$ such that $X-Y=\sum_{i=1}^r A_i$, i.e. the geodesic distance between $X$ and $Y$ considered as vertices of $\Gamma(\Alt_n(\F_q))$ is $r$.
\end{proof}

Recall the following formula for counting the number of alternating matrices of a given rank obtained \cite{ravagnani2018duality}.

\begin{lemma}\cite[Proposition 62]{ravagnani2018duality}\label{lem:rank_i}
    Let $0\leq i\leq n$ be an integer. It holds that 
    $$|\{M\in\Alt_n(\F_q) : \rk(M)=i\}| = \qqbin{n}{i}\sum_{s=0}^i(-1)^{i-s}q^{\binom{s}{2}+\binom{i-s}{2}}\qqbin{i}{s}.$$
\end{lemma}

\begin{proposition}\label{graphdegree}
    Let $\delta=\frac{(q^n-1)(q^{n-1}-1)}{q^2-1}$ be the number of alternating matrices in $\Alt_n(\F_q)$ with rank $2$. Then $\Gamma(\Alt_n(\F_q))$ is $\delta$-regular. 
\end{proposition}
\begin{proof}
    The rank distance is translation invariant, i.e. for every $X,Y,Z\in\Alt_n(\F_q)$ we have $\rk(X,Y)=\rk(X+Z,Y+Z)$. This implies that the number of neighbors of any $X\in\Gamma(\Alt_n(\F_q))$ is the same as the number of neighbors of the zero matrix $O$. Such number is the amount of alternating matrices of rank $2$. By applying the formula from Lemma \ref{lem:rank_i} we obtain that this is precisely $\delta$.
\end{proof}

Using Lemma \ref{lem:geo_distance} we can derive the following result, which we will use later on to apply the spectral bounds on the $k$-th independence number $\alpha_k$ to our graph $\Gamma(\Alt_n(\F_q)$.

\begin{lemma}\label{lemma:equivalence}
Let $n\geq d\geq 2$. Then $A_q(n,2d)=\alpha_{d-1}\big(\Gamma(\Alt_n(\F_q)) \big)$.
\end{lemma}
\begin{proof}
Let $\mathcal{A}$ be a subset of the set of vertices of $\Gamma(\Alt_n(\F_q))$ and assume that $\mathcal{A}$ is a $(d-1)$-independent set of maximal cardinality. Then, every pair of vertices in $\mathcal{A}$ have geodesic distance at least $d$. By Lemma~\ref{lem:geo_distance}, the rank distance of those matrices is at least $2d$. Hence, $A_q(n,2d)\geq |\mathcal{A}| = \alpha_{d-1}$. Vice versa, let $\mC\subseteq \Alt_n(\F_q)$ be an alternating code with minimum distance $2d$ and maximal cardinality. Consider the subgraph $\Gamma'(\mC)$ of $\Gamma(\Alt_n(\F_q)$ induced by the elements of $\mC$. Then the geodesic distance between every pair of vertices in $\Gamma'(\mC)$ is at least $d$ (since these are matrices with rank distance at least $2d$). Hence, $\alpha_{d-1}\geq |V(\Gamma'(\mC))|=|\mC| = A_q(n,2d)$, where $V(\Gamma'(\mC))$ denotes the set of vertices of $\Gamma'(\mC)$.
\end{proof}

Delsarte and Goethals in \cite{delsarte1975alternating} also computed the eigenvalues of the graph $\Gamma(\Alt_n(\F_q))$, which can be found as follows.

\begin{lemma}\label{graphevs}
The eigenvalues of the adjacency matrix of the graph $\Gamma(\Alt_n(\F_q))$ can be calculated as the evaluation of
\begin{equation}\label{eq:evs}
    P^{(n)}(x) = \frac{q^{2n-2x-1}-q^n-q^{n-1}+1}{q^2-1}, \; \textnormal{ for all } x\in\left\{0,1,\ldots, \left\lfloor\frac{n}{2}\right\rfloor \right\}.
\end{equation}
\end{lemma}

\begin{proof}
This is a straightforward consequence of Equation (15) in \cite{delsarte1975alternating}.
\end{proof}

Note that $P^{(n)}(0)=\delta$, as expected by Proposition \ref{graphdegree} and it is the largest eigenvalue of $\Gamma(\Alt_n(\F_q))$. The eigenvalues that we will use later are $P^{(n)}(\left\lfloor\frac{n}{2}\right\rfloor)$, $P^{(n)}(\left\lfloor\frac{n}{2}\right\rfloor-1)$, and $P^{(n)}(\left\lfloor\frac{n}{2}\right\rfloor-2)$. We write here their explicit value according to the parity of $n$. 
If $n$ is even, then 
\begin{align*}
    P^{(n)}\left(\frac{n}{2}\right)&=\frac{1-q^n}{q^2-1},\\
    P^{(n)}\left(\frac{n}{2}-1\right)&=\frac{q^{n+1}-q^n-q^{n-1}+1}{q^2-1},\\
    P^{(n)}\left(\frac{n}{2}-2\right)&=\frac{q^{n+3}-q^n-q^{n-1}+1}{q^2-1}.
\end{align*}
If $n$ is odd, then
\begin{align*}
    P^{(n)}\left(\frac{n-1}{2}\right)&=\frac{1-q^{n-1}}{q^2-1},\\
    P^{(n)}\left(\frac{n-1}{2}-1\right)&=\frac{q^{n+2}-q^n-q^{n-1}+1}{q^2-1},\\
    P^{(n)}\left(\frac{n-1}{2}-2\right)&=\frac{q^{n+4}-q^n-q^{n-1}+1}{q^2-1}.
\end{align*}

\begin{lemma}\label{lem:order_evs}
    For every integer $n\geq 2$ we have that $P^{(n)}(\left\lfloor\frac{n}{2}\right\rfloor)<-1$ and $P^{(n)}(x)>0$ for every $x\in\{0,1,\ldots,\left\lfloor\frac{n}{2}\right\rfloor-1\}$. Moreover, $P^{(n)}(x)$ is decreasing.
\end{lemma}
\begin{proof}
   It is easy to see that $P^{(n)}(x)$ is decreasing. For instance consider that the derivative $P^{(n)\prime}(x)=-\frac{2\log(q)q^{2n-2x-1}}{q^2-1}<0$ for every $x\in \left\{0,\ldots,\lfloor\frac{n}{2}\rfloor\right\}$. Moreover, $P^{(n)}(\left\lfloor\frac{n}{2}\right\rfloor)<-1$ for every $n\geq 2$. Finally, since $P^{(n)}((\left\lfloor\frac{n}{2}\right\rfloor -1)>0$ and $P^{(n)}(x)$ is decreasing, the statement follows.
\end{proof}

The intersection array of the alternating forms graph $\Gamma(\Alt_n(q))$ is also completely determined.

\begin{theorem}\cite[Theorem 9.5.6]{brouwer1989distance}\label{thm:drg}
    For $i\in\left\{0,\ldots,\lfloor\frac{n}{2}\rfloor\right\}$, the graph $\Gamma(\Alt_n(\F_q))$ is distance-regular with intersection array given by 
    \begin{align*}
        b_i &=\frac{q^{4i}(q^{n-2i}-1)(q^{n-2i-1}-1)}{q^2-1},\\
        c_i &=\frac{q^{2i-2}(q^{2i}-1)}{q^2-1},\\
        a_i&= \delta - b_i-c_i,
    \end{align*}
    where $\delta=\frac{(q^n-1)(q^{n-1}-1)}{q^2-1}$.
\end{theorem}

A linear algebra interpretation of the intersection numbers of $\Gamma(\Alt_n(\F_q))$ can be given as follows. For every $A\in\Alt_n(\F_q)$ with $\rk(A)=2i$, we have that 
\begin{align*}
    b_i&=|\{B\in\Alt_n(\F_q) \mid \rk(B)=2, \mbox{ and } \rk(A-B)=2(i-1)\}|,\\
    c_i&=|\{B\in\Alt_n(\F_q) \mid \rk(B)=2, \mbox{ and } \rk(A-B)=2(i+1)\}|,\\
    a_i&=|\{B\in\Alt_n(\F_q) \mid \rk(B)=2, \mbox{ and } \rk(A-B)=2i\}|.
\end{align*}

\begin{remark}
    For every field size $q$, when $n=4,5$, the adjacency matrix of the graphs $\Gamma(\Alt_4(\F_q))$ and $\Gamma(\Alt_5(\F_q))$ has exactly $3$ distinct eigenvalues. Moreover, the largest eigenvalue $\delta$ has multiplicity $1$. Thus, in this case these alternating bilinear forms graphs are strongly regular.
\end{remark}

Next we recall some eigenvalue bounds on the  $k$-independence number of a graph, which we will use from now onward and which will be compared with the Singleton-like bound in Eq.~\eqref{eq:SingletonBound}, provided by Delsarte and Goethals.

The first well-known result is the following bound due to Hoffman and never published; see for instance \cite{haemers1995interlacing}.

\begin{theorem}(Hoffman bound \cite{haemers1995interlacing})\label{thm:Hoffman}
 Let $\Gamma=(V,E)$ be a regular graph with $n$ vertices and adjacency eigenvalues $\lambda_1\geq\dots\geq\lambda_n$. Then 
$$\alpha(\Gamma) = \leq n\frac{-\lambda_n}{\lambda_1-\lambda_n}.$$
\end{theorem}

The above ratio bound was extended as follows.

\begin{theorem}(Ratio-type bound, \cite{ACF2019})\label{thm:hoffman-like} Let $\Gamma=(V,E)$ be a regular graph with $n$ vertices and adjacency eigenvalues $\lambda_1\geq\dots\geq\lambda_n$. Let $p\in\mathbb{R}_k[x]$, and define $W(p)=\max_{u\in V}\{(p(A))_{u,u}\}$ and $\lambda(p)=\min_{i=2,\dots,n}\{p(\lambda_i)\}$. 
Then $$\alpha_k(\Gamma)\leq n\frac{W(p)-\lambda(p)}{p(\lambda_1)-\lambda(p)}.$$
\end{theorem}

For a fixed $k$, applying the above Ratio-type bound is not straightforward. Indeed one needs to find a polynomial of degree $k$ which minimizes the right-hand side of the inequality above. In \cite{F2020}, an LP program which makes use of the so called \emph{minor polynomials} was proposed in order to find the polynomial $p$ that optimizes the above Ratio-type bound for a given $k$ and for a $k$-partially walk-regular graph $\Gamma$. In this LP program, the input are the distinct eigenvalues of the graph $\Gamma$, ordered as $\theta_0>\cdots>\theta_r$, with respective multiplicities $m(\theta_i)$, $i\in\{0,\dots,r\}$. The \textbf{minor polynomial} $f_k\in\R_k[x]$ is the polynomial that minimizes $\sum_{i=0}^r m(\theta_i) f_k(\theta_i)$. Let $p=f_k$ be defined by $f_k(\theta_0)=x_0=1$ and $f_k(\theta_i)=x_i$ for $i\in\{1,\dots,r\}$, where the vector $(x_1,\dots,x_r)$ is a solution of
the following linear program: 
\begin{equation}\label{LPminorpolynomials}
\boxed{
\begin{array}{ll@{}ll}
\text{minimize}  &\sum_{i=0}^{r} m(\theta_i)x_i &\\
\text{subject to} &f[\theta_0,\dots,\theta_s]=0, & \quad s=k+1,\dots,r\\
&x_i\geq0, &\quad i=1,\dots,r\\
\end{array}
}
\end{equation}
Here, $f[\theta_0,\dots,\theta_m]$ denote $m$-th divided differences of Newton interpolation, recursively defined by $$f[\theta_i,\dots,\theta_j]=\frac{f[\theta_{i+1},\dots,\theta_j]-f[\theta_i,\dots,\theta_{j-1}]}{\theta_j-\theta_i},$$ where $j>i$, starting with $f[\theta_i]=x_i$ for $i\in\{0,\dots,r.\}$.

Note that our graph $\Gamma(\Alt_n(\F_q))$ is distance-regular by Theorem \ref{thm:drg}, and therefore it holds the assumption of partially walk-regularity of the above LP.

For small $k$, some research has been devoted to find the best polynomial of degree $k$ which minimizes the right-hand side of the Ratio-type bound. In the case $k=1$, $\alpha_1$ is exactly the independence number, and by choosing $p(x)=x$ we get the original Hoffman's ratio bound; see Theorem \ref{thm:Hoffman}. For $k=2,3$ such polynomials have also been found in \cite{ACF2019,KN2022}, resulting in the following bounds.

\begin{theorem}[Ratio-type bound, \cite{ACF2019}]\label{thm:hoffman-k=2} Let $\Gamma$ be a regular graph with $n$ vertices and distinct eigenvalues of the adjacency matrix $\theta_0>\theta_1>\dots>\theta_r$ with $r\geq2$. Let $\theta_i$ be the largest eigenvalue such that $\theta_i\leq -1$. 
Then 
$$\alpha_2(\Gamma)\leq n\frac{\theta_0+\theta_i\theta_{i-1}}{(\theta_0-\theta_i)(\theta_0-\theta_{i-1})}.$$ 
Moreover, this is the best possible bound that can be obtained by choosing a polynomial and applying Theorem~\ref{thm:hoffman-like}.    
\end{theorem}

\begin{theorem}[Ratio-type bound, \cite{KN2022}]\label{thm:hoffman-k=3} Let $\Gamma$ be a regular graph with $n$ vertices and distinct eigenvalues of the adjacency matrix $\theta_0>\theta_1>\dots>\theta_r$ with $r\geq3$. Let $s$ be the largest index such that $\theta_s\geq -\frac{\theta_0^2+\theta_0\theta_r-\Delta}{\theta_0(\theta_r+1)}$, where $\Delta=\max_{u\in V}\{(A^3)_{u,u}\}$. Then
$$\alpha_3(\Gamma)\leq n\frac{\Delta-\theta_0(\theta_s+\theta_{s+1}+\theta_r)-\theta_s\theta_{s+1}\theta_r}{(\theta_0-\theta_s)(\theta_0-\theta_{s+1})(\theta_0-\theta_r)}.$$ Moreover, this is the best possible bound that can be obtained by choosing a polynomial and applying Theorem~\ref{thm:hoffman-like}.
\end{theorem}

By Lemma \ref{lemma:equivalence}, we have that $\alpha_{d-1}(\Gamma(\Alt_n(\F_q)))=A_q(n,2d)$. The rest of the section is devoted to compare the Ratio-type bounds for $\alpha$, $\alpha_2$ and $\alpha_3$ with the Singleton-like bound in Eq.~\eqref{eq:SingletonBound} for minimum distance value $2d$, with $d=2,3,4$, respectively.

\begin{proposition}\label{prop:Hoff=Singleton}
    Let $n\geq 4$ and consider $\theta_0> \theta_1 >\cdots >\theta_r$ to be the distinct adjacency eigenvalues of $\Gamma(\Alt_n(\F_q))$. Then the Hoffman bound from Theorem \ref{thm:Hoffman} and the Singleton-like bound in Eq.~\eqref{eq:SingletonBound} for minimum rank distance value equal to $4$ are equivalent.
\end{proposition}
\begin{proof}
    By Lemma \ref{lem:order_evs} we have that $\theta_0=\delta$ and $\theta_r=P^{(n)}\left(\left\lfloor\frac{n}{2}\right\rfloor\right)= -\frac{q^{2\left\lfloor\frac{n}{2}\right\rfloor}-1}{q^2-1}$.
    Plugging these into the Hoffman bound, we have that
    \begin{equation*}
        \alpha(\Gamma(\Alt_n(\F_q)))=q^{\frac{n(n-1)}{2}}\frac{-\theta_r}{\theta_0-\theta_r} = q^{\frac{n(n-1)}{2\left\lfloor\frac{n}{2}\right\rfloor}\left(\left\lfloor\frac{n}{2}\right\rfloor-1\right)},
    \end{equation*}
    which is equal to the Singleton-like bound in Eq.~\eqref{eq:SingletonBound} for minimum rank distance $4$.
\end{proof}

\begin{proposition}\label{prop:eqalpha2}
    Let $n\geq 6$ and let   $\theta_0>\theta_1>\dots>\theta_r$ be the eigenvalues of the adjacency matrix of $\Gamma(\Alt_n(\F_q))$. Let $\theta_i$ be the largest eigenvalue such that $\theta_i\leq -1$. Then bound from Theorem \ref{thm:hoffman-k=2} and the Singleton-like bound in Eq.~\eqref{eq:SingletonBound} for minimum rank distance $6$ are equivalent.
\end{proposition}
\begin{proof}
    By Lemma \ref{lem:order_evs} we have that $\theta_0=\delta$, $\theta_i=P^{(n)}(\left\lfloor\frac{n}{2}\right\rfloor)$ and $\theta_{i-1}=P^{(n)}(\left\lfloor\frac{n}{2}\right\rfloor-1)$. 
    Plugging these into the bound from Theorem \ref{thm:hoffman-k=2}, after tedious but straightforward computations, we get that 
    $$\alpha_2(\Gamma(\Alt_n(\F_q)))=q^{\frac{n(n-1)}{2}}\frac{\theta_0+\theta_i\theta_{i-1}}{(\theta_0-\theta_i)(\theta_0-\theta_{i-1})} = q^{\frac{n(n-1)}{2\left\lfloor\frac{n}{2}\right\rfloor}\left(\left\lfloor\frac{n}{2}\right\rfloor-2\right)},$$
    which gives precisely the Singleton-like bound for $d=3$.
\end{proof}

Let $A$ be the adjacency matrix of $\Gamma(\Alt_n(\F_q))$. In order to compute the quantity $\Delta=\max_{u\in V}\{(A^3)_{u,u}\}$ required to apply Theorem \ref{thm:hoffman-k=3}, we need the following folklore result.

\begin{lemma}
  Let $A$ be the adjacency matrix of a graph $\Gamma$. The $(i,j)$-th entry of $A^k$ counts the number of walks of length $k$ having start and end vertices $i$ and $j$ respectively. 
\end{lemma}

Since $\Gamma(\Alt_n(\F_q))$ is distance-regular, it is also walk-regular, meaning that all the entries on the main diagonal of $A$ are the same. Hence, it is sufficient to compute $(A^k)_{1,1}$, i.e. compute the number of walks of length $k$ having start and end vertex at the matrix $O\in\Alt_n(\F_q)$.

Hence we get immediately the following result.
\begin{lemma}\label{lem:Delta3}
    Let $A$ be the adjacency matrix of $\Gamma(\Alt_n(\F_q))$. Then 
    $$\Delta = \delta a_1=\frac{(q^n-1)(q^{n-1}-1)(q^{n+2}+q^{n+1}-q^n-q^{n-1}-q^4-q^2+2)}{(q^2-1)^2},$$
    where $a_1$ is computed in Theorem \ref{thm:drg}.
\end{lemma}

Let  $\delta=\theta_0>\theta_1>\dots>\theta_{\lfloor\frac{n}{2}\rfloor}$ be the eigenvalues of the adjacency matrix of $\Gamma(\Alt_n(\F_q))$. After some manipulations, we obtain
$$-\frac{\delta^2+\delta\theta_{\lfloor\frac{n}{2}\rfloor}-\Delta}{\delta(\theta_{\lfloor\frac{n}{2}\rfloor}+1)} = \frac{q^2(q^{n-2}-1)(q^{n-3}-1) - q^{2\lfloor\frac{n}{2}\rfloor-2}+1}{q^{2\lfloor\frac{n}{2}\rfloor-2}-1}=:t_n.$$

\begin{proposition}\label{propo:indexscondition}
      Let $n\geq 6$ and $\delta=\theta_0>\theta_1>\dots>\theta_{\lfloor\frac{n}{2}\rfloor}$ be the eigenvalues of the adjacency matrix of $\Gamma(\Alt_n(\F_q))$. Then the largest index $s\in\{0,\ldots,\left\lfloor\frac{n}{2}\right\rfloor\}$ such that $\theta_s\geq t_n$ is $\left\lfloor\frac{n}{2}\right\rfloor-2$.
\end{proposition}
\begin{proof}
    In order to prove the statement, we are going to show that 
    $$P^{(n)}\left(\left\lfloor\frac{n}{2}\right\rfloor-2\right) \geq t_n> P^{(n)}\left(\left\lfloor\frac{n}{2}\right\rfloor-1\right).$$
    Assume $n$ is even. Similar computations can be done when $n$ is odd.
    Recall that
    \begin{align*}
        P^{(n)}\left(\frac{n}{2}-2\right) &=\frac{q^{n+3}-q^n-q^{n-1}+1}{q^2-1}, \\
        P^{(n)}\left(\frac{n}{2}-1\right) &=\frac{q^{n+1}-q^n-q^{n-1}+1}{q^2-1}.
    \end{align*}
    Moreover $t_n = q^{n-1}-q^2-1$.
    Now let 
    \begin{align}
        P^{(n)}\left(\frac{n}{2}-2\right) - t_n &= \frac{q^{n+3}-q^n-q^{n-1}+1}{q^2-1} - q^{n-1}+q^2+1 \nonumber \\ 
        &=\frac{q^{n+3}-q^n -q^{n+1} + q^4}{q^2-1}. \label{eq:quantity}
    \end{align}
    The quantity in Eq.~\eqref{eq:quantity} is non-negative if and only if the numerator is non-negative. Hence
    \begin{align*}
        q^{n+3}-q^n -q^{n+1} + q^4 &\geq 0 \Leftrightarrow q^{n+1}\left(q^2 - \frac{1}{q} - 1\right) + q^4 \geq 0,
    \end{align*}
    which is always true.
    Hence, $P^{(n)}\left(\frac{n}{2}-2\right) \geq t_n$.
    Finally, let 
    \begin{align*}
        t_n - P^{(n)}\left(\frac{n}{2}-1\right) 
        &= \frac{q^n-q^4}{q^2-1},
    \end{align*}
    which is positive whenever $n\geq 6$.
    \end{proof}

The next theorem shows that the bound on $\alpha_3$ from Theorem \ref{thm:hoffman-k=3} equals the Singleton-like bound for minimum distance $8$. 

\begin{theorem}\label{thm:distance8SB}
    Let $n\geq 8$ and $s=\left\lfloor\frac{n}{2}\right\rfloor -2$. Let  $\delta=\theta_0>\theta_1>\dots>\theta_{\lfloor\frac{n}{2}\rfloor}$ be the distinct eigenvalues of the adjacency matrix of $\Gamma(\Alt_n(\F_q))$. Then the Ratio-type bound from Theorem \ref{thm:hoffman-k=3} and  and the Singleton-like bound in Eq.~\eqref{eq:SingletonBound} for minimum rank distance equal to $8$ are equivalent.
    \begin{proof}
        We prove the statement when $n$ is even. Similar computations can be done when $n$ is odd.
       We have the following identities.
        \small{\begin{align*}
            \theta_s+\theta_{s+1}+\theta_{\frac{n}{2}} &= P^{(n)}\left(\frac{n}{2}-2\right) +P^{(n)}\left(\frac{n}{2}-1\right) +P^{(n)}\left(\frac{n}{2}\right) \\
            &=\frac{1}{(q^2-1)^2}\Biggl(q^{3n-2}+q^{3n} -3q^{3n-1}-2q^{3n-2} -q^{2n+3} -q^{2n+2}-q^{2n+1}+2q^{2n}\\
            &\;\;+ 8q^{2n-1} + 2q^{2n-2}+q^{n+3} + q^{n+1} -6q^n-5q^{n-1} +3\Biggr).
        \end{align*}}
       \small{ \begin{align*}
        \theta_s\theta_{s+1}\theta_{\frac{n}{2}}=&\frac{1}{(q^2-1)^3}\Biggl(-q^{3n+4}+q^{3n+3}+q^{3n+2}+q^{3n+1} -2q^{3n-1} - q^{3n-2} +q^{2n+4}-2q^{2n+3}\\
        &-q^{2n+2}-2q^{2n+1}+2q^{2n}+4q^{2n-1} +q^{2n-2}+q^{n+3}+q^{n+1}-3q^n-2q^{n-1}+1\Biggr).
        \end{align*}}
        \small{\begin{align*}
            (\theta_0-\theta_s)(\theta_0-\theta_{s+1})\left(\theta_0-\theta_{\frac{n}{2}}\right) &=\frac{(q^{2n-1}-q^{n+3})(q^{2n-1}-q^{n+1})(q^{2n-1}-q^{n-1})}{(q^2-1)^3}.
        \end{align*}}
        Using these quantities inside the bound from Theorem \ref{thm:hoffman-k=3}, we get
    $$\alpha_3(\Alt_n(\F_q)))=q^{\frac{n(n-1)}{2}}\frac{\Delta-\theta_0\left(\theta_s+\theta_{s+1}+\theta_{\left\lfloor\frac{n}{2}\right\rfloor}\right)-\theta_s\theta_{s+1}\theta_{\left\lfloor\frac{n}{2}\right\rfloor}}{(\theta_0-\theta_s)(\theta_0-\theta_{s+1})\left(\theta_0-\theta_{\left\lfloor\frac{n}{2}\right\rfloor}\right)} = q^{\frac{n(n-1)}{2\left\lfloor \frac{n}{2}\right\rfloor}\left(\left\lfloor\frac{n}{2}\right\rfloor -3\right)},$$
    where $\Delta$ is computed in Lemma \ref{lem:Delta3}.
    \end{proof}
\end{theorem}

We have proved that for minimum distance value equal to $4$, $6$ and $8$, the Singleton-like bound given by Delsarte and Goethals and the Ratio-type bounds are equivalent. For larger minimum distance, we experimentally observed that the Ratio-type bound (Theorem \ref{thm:hoffman-like}) and its corresponding LP implementation (\ref{LPminorpolynomials}) coincide with the bound from Theorem \ref{thm:del-goeth}. In fact, our proposed eigenvalue Ratio-type bound runs much faster because it does not need the whole graph $\Gamma(\Alt_n(\F_q))$ as an input, but just its adjacency spectrum, which can be obtained using Lemma \ref{graphevs}.

\section{Coding-theoretic bounds}\label{sec:newcodingbounds}
In this section, we provide an overview of a number of coding theory methods to derive bounds on the parameters of alternating codes, such as Code-Anticode, Sphere Packing and Total Distance bounds. However, our results illustrate that such coding methods yield worse bounds than the Singleton-like one from \cite{delsarte1975alternating}, except for few parameters for which they give the same result.

\subsection{Code-Anticode Bound}
We start by recalling what an anticode is.

\begin{definition}
    A set $\mA\subseteq \Mat_{n}(\F_q)$ with the property that
    $\rk(M-N)\le r$
    for all $M,N\in\mA$
    is called an $r$-\textbf{anticode}. If $\mA$ is in addition an $\F_q$-subspace of $\Mat_{n}(\F_q)$, then it is called a \textbf{linear $r$-anticode}. 
\end{definition}
We also recall that the dimension of  $\Alt_{n}(\Fq)$ is $n(n-1)/2$.

In this subsection we only work with linear codes and anticodes: from now on, we omit the word ``linear".
The following result is an application of the celebrated Code-Anticode bound by Delsarte; see \cite{delsarte1998association}.

\begin{proposition}
    Let $\mC \subseteq \Alt_{n}(\F_q)$ be a code with minimum distance $d$ and let $\mA \subseteq \Alt_{n}(\F_q)$ be a $(d-1)$-anticode.
    Then     \begin{equation}\label{eq:code-anticode}
        |\mC| \cdot |\mA|\leq q^{n(n-1)/2}.
    \end{equation}
\end{proposition}

The bound in Eq.~\eqref{eq:code-anticode} is known as the \textbf{Code-Anticode} bound and we can use it to derive some results on the dimension of an alternating code. 
For a subspace $U\subseteq\F_q^n$,  define 
$$ \Alt_{n}(U):=\{M\in \Alt_{n}(\F_q) \mid \colsp(M)\subseteq U \}.$$

Note that $\Alt_{n}(U)$ is an anticode, since it is a linear space and the rank of all the matrices contained in it is upper bounded by the dimension of $U$.

\begin{lemma}
Let $U\subseteq \F_q^n$, be any $t$-dimensional subspace. Then the dimension of  $\Alt_{n}(U)$ is $t(t-1)/2$.
\end{lemma}
\begin{proof}
Let $\mB_U:=\{u_1,\ldots,u_t\}$ be a basis of $U$ and $\phi:\F_q^n\to\F_q^n$ be an $\F_q$-linear isomorphism
that sends $U$ into the space spanned by $e_1,\ldots, e_t$. Let $B\in\GL_n(\F_q)$ be the matrix associated with $\phi$ with respect to the canonical basis.
Since $B$ is invertible, we have that the map $M\mapsto BMB^\top$ is a rank-preserving $\F_q$-linear isomorphism $\Alt_{n}(U)\to \Alt_{n}(\langle e_1,\ldots, e_t\rangle_{\F_q})$. 

\end{proof}

We can now apply the previous result in combination with the Code-Anticode Bound.

\begin{corollary}
Let $\mC\subseteq \Alt_{n}(\F_q)$ be an alternating $k$-dimensional code with minimum rank distance $2d$. We have
\begin{equation}\label{eq:Sing_Anticode}
    k\leq \frac{n(n-1)-2(d-1)(2d-1)}{2}.
\end{equation}
\end{corollary}

\begin{proof}
 Let $U\subseteq \F_q^n$ be a $(2d-1)$-dimensional subspace and let $\mA=\Alt_{n}(U)$. Apply the Code-Anticode Bound with $\mC$ and $\mA$.
\end{proof}

\begin{remark}
    A comparison between the Singleton-like bound in Eq.~\eqref{eq:SingletonBound} and the bound in Eq.~\eqref{eq:Sing_Anticode}, shows that the first one is always better than the second one. The two bounds are equal if and only $n=2d$, which implies $k=1$ and therefore is a trivial case.
\end{remark}

\subsection{Sphere Packing Bound}

For $0 \le r \le n$
let 
$$\mB(\Alt_n(\F_q),r):=\{M \in \Alt_n(\F_q) \mid \rk(M) \le r\}$$
be the (\textbf{alternating}) \textbf{ball} of radius $r$ centered at $0$.
In \cite[Corollary 6.14]{gruica2022rank}, it has been computed the asymptotic volume of $\mB(\Alt_n(\F_q),r)$ as $q$ tends to infinity.

\begin{lemma}\cite{gruica2022rank} \label{lem:size_ball}
Let $0 \le r \le n$ be an integer. We have 
\begin{align*}
    |\mB(\Alt_n(\F_q),r)| \sim \begin{cases}
    q^{rn-r(r+1)/2} \quad &\textnormal{ if $r$ is even,} \\
    q^{(r-1)n-(r-1)r/2} \quad &\textnormal{ if $r$ is odd,}
    \end{cases} \quad \textnormal{ as $q \to +\infty$.}
\end{align*}
\end{lemma}

The following bound follows from a classical sphere packing argument.

\begin{proposition}
    An alternating code $\mC\subseteq \Alt_n(\F_q)$ with $|\mC|\geq 2$ and minimum distance $2d$ satisfies   
    \begin{equation}\label{eq:sphere_packing}
        |\mC|\leq \left\lfloor\frac{|\Alt_n(\F_q)|}{|\mB(\Alt_n(\F_q),t)|}\right\rfloor,
    \end{equation}
    where $t=\left\lfloor\frac{2d-1}{2}\right\rfloor$.
\end{proposition}

We can quite easily show that there are no nontrivial perfect alternating codes for $q$ sufficiently large.

\begin{proposition}\label{prop:No_perfect_asymptotic}
    There are no nontrivial perfect alternating rank-metric codes in $\Alt_n(\F_q)$, for $q$ sufficiently large.
\end{proposition}
\begin{proof}
   Suppose that there exists a perfect alternating code in $\Alt_n(\F_q)$ with minimum distance $2d$. Define $t=\left\lfloor\frac{2d-1}{2}\right\rfloor=d-1$. Suppose that $t$ is even, i.e., that $d$ is odd.
   Then from the Sphere Packing bound in Eq.~\eqref{eq:sphere_packing} and Lemma \ref{lem:size_ball} we have that 
   \begin{align}\label{eq:SB+SP}
       |\mC|\cdot q^{tn-\frac{t(t+1)}{2}} = q^{\frac{n(n-1)}{2}}.
   \end{align}
   By the Singleton-like bound. We have that 
   $$   |\mC|\leq \begin{cases}
       q^{n\left(\frac{n-2d+1}{2}\right)} &\textnormal{ if } n \textnormal{ is odd}, \\
        q^{(n-1)\left(\frac{n-2d+2}{2}\right)}   &\textnormal{ if } n \textnormal{ is even}.
   \end{cases}$$
   If $n$ is odd, by combining the Singleton-like bound with Eq.~\eqref{eq:SB+SP}, we get
   \begin{align*}
        n\left(\frac{n-2t-1}{2}\right) + tn-\frac{t(t+1)}{2} = \frac{n(n-1)}{2},
   \end{align*}
   which holds only for $t=0$ or $t=-1$. Since $t\geq 0$, the only admissible possibility is $t=0$, so $d=1$. In this case we get the whole space, which is a trivial case.
   The cases $n,t$ even and $t$ odd are similar.
\end{proof}

We can also show that there are no perfect alternating rank-metric codes for any value of $q$ when $d$ is even.

\begin{proposition}\label{prop:perfect_d_even}
    For any $q$, there are no perfect alternating rank-metric codes with minimum distance $2d$ in $\Alt_n(\F_q)$, for an even values $d$.
\end{proposition}
\begin{proof}
    Let $\mC\subseteq \Alt_n(\F_q)$ be an alternating rank-metric codes with minimum distance $2d$ and let $t=d-1$.
    Consider the packing given by the balls of radius $t$ centered at the codewords of $\mC$. Let $A\in\mC$ be a codeword of rank $2d$. Then by Lemma \ref{lem:sumrank2}, $A$ can be written as the sum of $d$ alternating matrices of rank $2$. Consider $B$ to be the sum of $d/2$ among these matrices of rank $2$. Then $B$ is an alternating matrix of rank $d$, which cannot be covered by balls of radius $d-1$ centered at codeword of $\mC$. This shows that the considered packing does not cover the whole space $\Alt_n(\F_q)$, hence $\mC$ cannot be perfect.
\end{proof}

\subsection{Total Distance Bound}
We can estimate the \emph{total distance} of a code $\mC$, i.e. the sum of the rank distances between the codewords of an $\mC$ in order to give another bound on the cardinality of $\mC$. The proof of this bound is inspired by \cite{byrne2021fundamental}.

\begin{proposition}\label{tot_dis_bound}
    Let $\mC\subseteq\Alt_n(\F_q)$ be an alternating code with minimum distance $2d$. Then 
    $$d\leq \begin{cases}
        \left\lfloor\frac{n}{2}\right\rfloor +\frac{1-q^{1-n}|\mC|}{|C|-1} & \textnormal{ if } n \textnormal{ is even,}\\
        \left\lfloor\frac{n-1}{2}\right\rfloor +\frac{1-q^{3-2n}|\mC|}{|\mC|-1} & \textnormal{ if } n \textnormal{ is odd}.
    \end{cases}$$
    In particular, if $n$ is even and $d>\frac{n}{2} -q^{1-n}$, then
    \begin{equation*}\label{eq:tot_dis_bound}
        |\mC|\leq \frac{d-\frac{n}{2}+1}{d-\frac{n}{2} +q^{1-n}};
    \end{equation*}
    if $n$ is odd and $d>\frac{n-1}{2} -q^{3-2n}$, then
    \begin{equation*}\label{eq:tot_dis_boundnodd}
        |\mC|\leq \frac{d-\frac{n-1}{2}+1}{d-\frac{n-1}{2} +q^{3-2n}}.
    \end{equation*}
\end{proposition}
\begin{proof}
    First of all define the ``total distance'' of $\mC$ as
$$S(\mC)=\sum_{\substack{X,Y\in\mC\\ X\ne Y}}\rk(X-Y).$$
    Since for every $X,Y\in\mC$ with $X \ne Y$ we have that $\rk(X-Y)\geq 2d$, it holds
\begin{equation}\label{eq:LBonS}
        S(\mC)\geq 2d|\mC|(|\mC|-1). 
    \end{equation}
    Consider the two quantities
    \begin{align*}
        S_1&:=|\{(X,Y)\in\mC^2 : \rk(X,Y)=2\left\lfloor\frac{n}{2}\right\rfloor\}|,\\
        S_2&:=|\{(X,Y)\in\mC^2 : \rk(X,Y)\leq 2\left\lfloor\frac{n}{2}\right\rfloor-2\}|.
    \end{align*}
    Then, since $X\neq Y$, we have $S_1+S_2=|\mC|(|\mC|-1)$.
    We have 
    \begin{align}
        S(\mC)&\leq 2\left\lfloor\frac{n}{2}\right\rfloor S_1 + \left(2\left\lfloor\frac{n}{2}\right\rfloor-2\right)S_2\nonumber \\
        &=2\left\lfloor\frac{n}{2}\right\rfloor |\mC|(|\mC|-1) -2S_2. \label{eq:UBonS}
    \end{align}
    In order to derive a lower bound on $S_2$, we distinguish two cases. 
   If $n$ is even, observe that if two codewords $X,Y\in\mC$ have the same first row, then their difference cannot have full rank, that is, $\rk(X-Y)\ne n$. For $X\in\mC$ denote by $X^1$ the first row of $X$.
    Hence, for $A\in\F_q^{1\times n}$ define 
    $$N_A:=|\{X\in\mC : X^1 =A\}|.$$
    We have that 
    $$\sum_{\substack{A\in\F_q^{1\times n}\\A_{1,1}=0}}N_A=|\mC|, \qquad \textnormal{ and } \qquad S_2\geq \sum_{\substack{A\in\F_q^{1\times n}\\A_{1,1}=0}}N_A(N_A-1).$$
    Combining Eq.~\eqref{eq:LBonS}, Eq.~\eqref{eq:UBonS} and the above inequality, we obtain
    \allowdisplaybreaks
    \begin{align*}
        2d|\mC|(|\mC|-1)\leq S(\mC)&\leq n |\mC|(|\mC|-1) -2S_2\\ 
        &\leq n|\mC|(|\mC|-1) -2\sum_{\substack{A\in\F_q^{1\times n}\\A_{1,1}=0}}N_A^2 +2\sum_{\substack{A\in\F_q^{1\times n}\\A_{1,1}=0}}N_A\\
        &=n |\mC|(|\mC|-1) -2\sum_{\substack{A\in\F_q^{1\times n}\\A_{1,1}=0}}N_A^2 +2|\mC|\\
        &\leq n|\mC|(|\mC|-1)+2|\mC|-2|\mC|^2q^{1-n}, 
    \end{align*}
    where the last inequality follows from the Cauchy–Schwarz inequality. The first bound in the statement follows by  rearranging the terms.

   If $n$ is odd, we slight modify the argument above. Indeed, in this case we note that if two codewords $X,Y\in\mC$ have the same first two rows, then their difference cannot have full rank, that is, $\rk(X-Y)\ne n-1$. For $X\in\mC$ denote by $X^{1,2}$ the first two rows of~$X$.
    Hence, for $A\in\F_q^{2\times n}$ define 
    $$N_A:=|\{X\in\mC : X^{1,2} =A\}|.$$
   Now we proceed as before:    $$\sum_{\substack{A\in\F_q^{2\times n}\\A_{1,1}=A_{2,2}=0\\A_{2,1}=-A_{1,2}}}N_A=|\mC|, \qquad \textnormal{ and } \qquad S_2\geq \sum_{\substack{A\in\F_q^{2\times n}\\A_{1,1}=A_{2,2}=0\\A_{2,1}=-A_{1,2}}}N_A(N_A-1).$$
   Hence we obtain
    \begin{align*}
        2d|\mC|(|\mC|-1)\leq S(\mC)&\leq (n-1)|\mC|(|\mC|-1) -2S_2\\ 
        &\leq (n-1)|\mC|(|\mC|-1) -2\sum_{\substack{A\in\F_q^{2\times n}\\A_{1,1}=A_{2,2}=0\\A_{2,1}=-A_{1,2}}}N_A^2 +2\sum_{\substack{A\in\F_q^{2\times n}\\A_{1,1}=A_{2,2}=0\\A_{2,1}=-A_{1,2}}}N_A\\
        &=(n-1) |\mC|(|\mC|-1) -2\sum_{\substack{A\in\F_q^{2\times n}\\A_{1,1}=A_{2,2}=0\\A_{2,1}=-A_{1,2}}}N_A^2 +2|\mC|\\
        &\leq (n-1)|\mC|(|\mC|-1)+2|\mC|-2|\mC|^2q^{3-2n}, 
    \end{align*}
    where the last inequality follows from the Cauchy–Schwarz inequality. The bound again follows by rearranging the terms.
\end{proof}

\begin{remark}

    Observe that when $n$ is even the bound from Proposition~\ref{tot_dis_bound} holds only if $d>\frac{n}{2}-q^{1-n}\geq \frac{n}{2}-1$, while when $n$ is odd the bound holds only if $d>\frac{n-1}{2}-q^{3-2n}\geq \frac{n-1}{2}-1$. Since $d\leq \left\lfloor\frac{n}{2}\right\rfloor$, we have that the Total Distance bound gives admissible results only when $d=\left\lfloor\frac{n}{2}\right\rfloor$. In this case we get  in fact 
    $$A_q\left(n,2\left\lfloor\frac{n}{2}\right\rfloor\right)\leq \begin{cases}
        q^{n-1} & \textnormal{if } n \textnormal{ is even,}\\
        q^{2n-3} & \textnormal{if } n \textnormal{ is odd}.
    \end{cases}$$ 
    This upper bound coincides with the Singleton-like bound in Eq.~\eqref{eq:SingletonBound} for $n$ even.  On the other hand, when $n$ is odd the two bounds coincide only if $n=3$.
\end{remark}

The total distance bound  states that, when $n$ is even, the dimension of a linear alternating code $C\subseteq \Alt_n(\F_q)$ is at most $n-1$ whenever the minimum distance of the code $\mC$ is $n$.

The next result shows that Delsarte and Goethals' Singleton-like bound is far from being sharp in case of linear alternating rank-metric codes.

\begin{proposition}\cite[Lemma 3]{GQ2006}
Let $n$ even, $q$ odd and $\mC\subseteq \Alt_n(\F_q)$ with minimum distance $2d=n$, then 
$$ \dim(\mC) \leq \frac{n}{2}.$$
\end{proposition}

Previous lemma suggests that there is space for improving the bound at least for the case of linear codes.

\section{Concluding remarks}\label{sec:conclusion}

We proposed a method based on eigenvalue bounds as a simpler alternative to bound the size of alternating rank-metric codes, avoiding the use of the corresponding association scheme needed to apply Delsarte's LP approach, as it was classically done in \cite{delsarte1975alternating}. In fact, the best  Ratio-type eigenvalue bounds can also be computed in polynomial time via closed formulas for small minimum distances, and via an LP for general minimum distance, similarly as the LP bound corresponding to the scheme for bilinear alternating forms proposed by Delsarte and Goethals in \cite{delsarte1975alternating}. Computational experiments show that the proposed eigenvalue bounds are equally good as the bound proposed in \cite{delsarte1975alternating}. In fact, we can prove that both algebraic approaches are equivalent to the corresponding Singleton-like bound derived in \cite{delsarte1975alternating} for small values of the minimum distance.
It would be interesting to prove the equivalence between the Ratio-type LP and Delsarte LP approaches for general minimum distance and alternating codes. Moreover, given the equality, it would be also interesting to use the explicit bound in order to derive an upper bound on $\alpha_4(\Gamma(\Alt_n(\F_q)))$.

\bigskip
\subsection*{Acknowledgements}
Aida Abiad is supported by the Dutch Research Council through the grant VI.Vidi.213.085. Gianira N. Alfarano is supported by he Swiss National Foundation through the grant no. 210966. Alberto Ravagnani is supported by the Dutch Research Council through the grants VI.Vidi.203.045, OCENW.KLEIN.539, and by the Royal Academy of Arts and Sciences of the Netherlands.

\bibliographystyle{abbrv}
\bibliography{references}

\end{document}